\documentclass[12pt,reqno]{amsart}
\usepackage{latexsym,amsthm,amsmath,amssymb,mathrsfs,mathtools,xcolor}
\usepackage[T1]{fontenc}
\usepackage[utf8]{inputenc}
\usepackage[mathscr]{eucal}
\usepackage{enumerate}

\usepackage[normalem]{ulem}

\usepackage{tikz}
\usetikzlibrary{arrows,calc,patterns,cd}
\usepackage{wrapfig}

\def\mvint_#1{\mathchoice
          {\mathop{\vrule width 6pt height 3 pt depth -2.5pt
                  \kern -9pt \intop}\limits_{\kern -3pt #1}}%
          {\mathop{\vrule width 5pt height 3 pt depth -2.6pt
                  \kern -6pt \intop}\nolimits_{#1}}%
          {\mathop{\vrule width 5pt height 3 pt depth -2.6pt
                  \kern -6pt \intop}\nolimits_{#1}}%
          {\mathop{\vrule width 5pt height 3 pt depth -2.6pt
                  \kern -6pt \intop}\nolimits_{#1}}}

\newcommand{\R}{\mathbb R}

\newtheorem{theorem}{Theorem}
\newtheorem*{theorem*}{Theorem}

\newtheorem{corollary}[theorem]{Corollary}

\theoremstyle{definition}
\newtheorem{remark}[theorem]{Remark}
\newtheorem*{remark*}{Remark}
\newtheorem{example}[theorem]{Example}

\usepackage{setspace}

\title[A note on spaces supporting Poincar\'e inequalities]{A note on metric-measure spaces supporting Poincar\'e inequalities}
\author[Alvarado]{Ryan Alvarado}
\address{Ryan Alvarado,\newline \indent Department of Mathematics and Statistics, Amherst College, 
\newline \indent 502 Seeley Mudd, Amherst,
Massachusetts 01002}
\email{rjalvarado\@@amherst.edu}

\author[Haj\l{}asz]{Piotr Haj\l{}asz}
\address{Piotr Haj\l{}asz,\newline \indent Department of Mathematics, University of Pittsburgh, \newline \indent 301 Thackeray Hall, Pittsburgh,
Pennsylvania 15260}
\email{hajlasz@pitt.edu}
\thanks{P.H. was supported by NSF grant DMS-1800457.}


\begin{document}
\maketitle
\sloppy

\begin{center}
{\em Dedicated to Professor Vladimir Maz'ya on the occasion of his 80th birthday.}
\end{center}

\begin{abstract}
Using a method of Korobenko, Maldonado and Rios we show a new characterization of doubling metric-measure spaces supporting Poincar\'e inequalities without assuming a priori that the measure is doubling.
\end{abstract}

Non-smooth functions have played a key role in analysis since the nineteenth century. One fundamental development in this vein came with the introduction of Sobolev spaces, which turned out to be a key tool in studying nonlinear partial differential equations and calculus of variations. Although classically Sobolev functions themselves were not smooth, they were defined on smooth objects such as domains in the Euclidean space or, more generally, Riemannian manifolds. By the late 1970s it became well recognized that several results in real analysis required little structure from the underlying ambient space, and could be generalized to non-smooth settings,
such as to the so-called spaces of homogeneous type. The latter spaces are (quasi)metric spaces equipped with a doubling Borel measure (see \cite{CoWe71,CoWe77}). In fact, maximal functions, Hardy spaces, functions of bounded mean oscillation, and singular integrals of Calder\'on-Zygmund-type all continue to have a fruitful theory in the context of spaces of homogeneous type. However, this rich theory was, in a sense, only zeroth-order analysis given that no derivatives were involved. The study of first-order analysis with suitable generalizations of derivatives, a fundamental theorem of calculus, and Sobolev spaces, in the setting of spaces of homogeneous type, was initiated in the 1990s. This area, known as analysis on metric spaces, has since grown into a multifaceted theory which continues to play an important role in many areas of contemporary mathematics.
For an introduction to the subject we recommend \cite{AlMi,AGS,AT,BB,cheeger,hajlasz,SMP,HeinonenK,HKST,shanmugalingam}.

One of the main objects of study in analysis on metric spaces are so called spaces supporting Poincar\'e inequalities introduced in \cite{HeinonenK}. To define this notion, recall that a {\em metric-measure space} $(X,d,\mu)$ is a metric space $(X,d)$ with a Borel measure $\mu$ such that $0<\mu\big(B(x,r)\big)<\infty$ for all $x\in X$ and all $r\in(0,\infty)$, where $B(x,r)$ denotes the (open) metric ball with center $x$ and radius $r$, i.e., $B(x,r):=\{y\in X:\, d(x,y)<r\}$. If the measure $\mu$ is {\em doubling}, that is, if
there exists a finite constant $C>0$ such that
$\mu(2B)\leq C\mu(B)$ for all balls $B\subseteq X$, then we call $(X,d,\mu)$ a {\em doubling metric-measure space}.
The notation $\tau B$ denotes the dilation of a ball $B$ by a factor $\tau\in(0,\infty)$, i.e., $\tau B:=B(x,\tau r)$.
A Borel function $g:X\to[0,\infty]$ is said to be an {\em upper gradient} of another Borel function $u:X\to\mathbb{R}$ if
\begin{equation}
|u(x)-u(y)|\leq\int_{\gamma_{xy}} g\,ds,
\end{equation}
holds for each $x,y\in X$ and all rectifiable curves $\gamma_{xy}$ joining $x,y$. Finally, a  metric-measure space $(X,d,\mu)$ is
said to {\em support a $p$-Poincar\'e  inequality}, $p\in[1,\infty)$, if there exist
constants $C\in(0,\infty)$ and $\sigma\in[1,\infty)$
such that
\begin{equation}
\label{ppoin}
\mvint_{B} |u-u_{B}|\,d\mu\leq Cr
\left(\,\, \mvint_{\sigma B}g^{p}\, d\mu\right)^{1/p},
\end{equation}
whenever $B$ is a ball of radius $r\in(0,\infty)$,
$u\in L^1_{\rm loc}(X,\mu)$, and $g:X\to[0,\infty]$   is an upper gradient of $u$. 
Here and in what follows the barred integral and $f_E$ stand for the integral average:
$$
f_E=\mvint_Ef\, d\mu =\frac{1}{\mu(E)}\int_E f\, d\mu,
$$
where $E$ is a $\mu$-measurable set of positive measure.
To be consistent with the definition of the upper gradient, in what follows we will always assume that functions $u\in L^1_{\rm loc}(X,\mu)$ are 
everywhere finite Borel representatives.
The above definitions of the upper gradient and spaces supporting Poincar\'e inequalites
are due to Heinonen and Koskela in \cite{HeinonenK} (see also \cite{HKST} for a more detailed exposition).

It was proved in \cite{SMP2} and \cite[Theorem~5.1]{SMP} that if a doubling metric-measure space supports a Poincar\'e inequality, then
the $p$-Poincar\'e inequality
self-improves in the sense that for some $q\in(p,\infty)$ and
$C'\in(0,\infty)$, there holds
\begin{equation}
\label{REW-1}
\left(\, \mvint_{B} |u-u_{B}|^{q}\, d\mu\right)^{1/q}\leq
C'r
\left(\,\, \mvint_{5\sigma B}g^{p}\, d\mu\right)^{1/p},
\end{equation}
whenever $B$ is a ball of radius $r\in(0,\infty)$,
$u\in L^1_{\rm loc}(X,\mu)$, and $g:X\to[0,\infty]$  is an upper gradient of $u$. 
%
%

The purpose of this note is to show that the family of inequalities in \eqref{REW-1} on a metric measure space imply that the underlying measure is doubling, and thus providing a characterization of doubling metric-measure spaces supporting Poincar\'e inequalities without assuming a priori that the measure is doubling, see Theorem~\ref{EPequiv}, below. This result is a minor  refinement of a beautiful result in \cite{korobenkomr}, where it was proved that in a related context, a family of weak Sobolev inequalities imply that the measure is doubling. However, the authors considered Sobolev inequalities where the balls had the same radius on both sides, and such a condition is stronger than the one in \eqref{REW-1}. Moreover, they did not address the important applications to Sobolev spaces supporting Poincar\'e inequalities. 

While the proof presented below is almost the same as the one in \cite{korobenkomr}, it is important to provide details:\ the proof employs an infinite iteration of Sobolev inequalities and since now we have balls of different size on both sides, it is not obvious without checking details that this will not cause estimates to blow up. This paper should be regarded as a supplement to the work of \cite{korobenkomr} and an advertisement of their work. 
Different, but related iterative arguments to the one presented below were used in \cite{carron,gorka,hajlaszkt1,hajlaszkt2,hebey} in the proofs that a Sobolev inequality implies a measure density condition. Other applications of a method developed in \cite{korobenkomr} are given in \cite{AGH,korobenko}.

We now state the main result of this note.
\begin{theorem}
\label{EPequiv}
Let $(X,d,\mu)$ be a metric-measure space and fix $p\in[1,\infty)$.
Then the following two statements are equivalent.
\begin{enumerate}
\item[(a)] The measure $\mu$ is doubling and the space $(X,d,\mu)$
supports a $p$-Poincar\'{e} inequality.
\vskip.08in

\item[(b)] There exist $q\in(p,\infty)$,  $C_P\in [1,\infty)$, and $\sigma\in[1,\infty)$
such that
\begin{equation}
\label{eq20}
\left(\, \mvint_{B} |u-u_{B}|^{q}\, d\mu\right)^{1/q}\leq
C_Pr
\left(\,\, \mvint_{\sigma B}g^{p}\, d\mu\right)^{1/p},
\end{equation}
whenever $B$ is a ball of radius $r\in(0,\infty)$, $u\in L^1_{\rm loc}(X,\mu)$,
and $g:X\to[0,\infty]$ is an upper gradient of $u$.
\end{enumerate}
\end{theorem}
\begin{remark}
We could assume that \eqref{eq20} holds with $C_P\in (0,\infty)$, but the estimates presented below
are more elegant if $C_P\geq 1$. Clearly, if \eqref{eq20} holds with a constant strictly greater than zero, then we can increase it to a constant greater than or equal to $1$.
\end{remark}

A positive locally integrable function $0<w\in L^1_{\rm loc}(\R^n)$ defines an absolutely continuous measure 
$d\mu=w(x)\, dx$ with the weight $w$. A class of the so called {\em $p$-admissible} weights plays a fundamental role in 
the nonlinear potential theory \cite{HKM}. To make the presentation brief, we will not recall the definition of a $p$-admissible weight, but we refer the reader to \cite{HKM} for details. 
As an immediate consequence of Theorem~\ref{EPequiv} and \cite[Theorem~2]{SMP2} we obtain a new characterization of $p$-admissible weights. A variant of this result has also been proved in \cite{korobenkomr}.
\begin{corollary}
A function $0<w\in L^1_{\rm loc}(\R^n)$ is a $p$-admissible weight for some $1<p<\infty$, if and only if 
there exist $q\in (p,\infty)$, $C\in [1,\infty)$ and $\sigma\in [1,\infty)$ such that
$$
\left(\,\mvint_B |u-u_B|^q\, d\mu\right)^{1/q}\leq
Cr\left(\,\mvint_{\sigma B} |\nabla u|^p\, d\mu\right)^{1/p}
$$
whenever $B\subset\R^n$ is a ball of radius $r\in (0,\infty)$, $u\in C^\infty(\sigma B)$ and $d\mu=w\, dx$.
\end{corollary}

\begin{example}
\label{rem}
The following example from \cite[Example~6.2]{BBJ}  shows a metric-measure space with a non-doubling measure that supports the $(p,p)$-Poincar\'e inequality \eqref{eq21} for all $1\leq p<\infty$. This shows that Theorem~\ref{EPequiv} is sharp in the sense that $q>p$ in \eqref{eq20} cannot be replaced by $q=p$. Let $X=[0,\infty)$ be equipped with the Euclidean metric $d(x,y)=|x-y|$ and the measure $d\mu=w(x)\, dx$, where $w(x)=\min\{1,x^{-1}\}$. First observe that $w$ is not doubling. Indeed, for $r>1$, $\mu(B(2r,r))=\mu((r,3r))=\int_r^{3r} x^{-1}\, dx=\log 3$, but
$\mu(2B(2r,r))=\mu((0,4r))>\int_1^{4r} x^{-1}=\log (4r)\to\infty$ as $r\to\infty$. It remains to show the $(p,p)$-Poincar\'e inequality for $1\leq p<\infty$.

Any ball $B=B(x,r)$ is an interval $(a,b)$ or $[a,b)=[0,b)$ with $r\leq b-a\leq 2r$. Since the function $w$ is non-increasing, for any $x\in [a,b)$ we have
$$
\int_x^b w(t)\, dt\leq (b-x)w(x)\leq 2r w(x).
$$
This observation and H\"older's inequality yield
\begin{equation*}
\begin{split}
\frac{1}{\mu(B)} \int_B |u-u(a)|^p\, d\mu
&\leq
\frac{1}{\mu(B)}\int_a^b \left|\int_a^t g(x)\, dx\right|^p w(t)\, dt\\
&\leq
\frac{(2r)^{p-1}}{\mu(B)}
\int_a^b\left(\int_a^t g(x)^p\, dx\right) w(t)\, dt\\ 
&=
\frac{(2r)^{p-1}}{\mu(B)}
\int_a^b\left(\int_x^b w(t)\, dt\right) g(x)^p\, dx\\
&\leq 
\frac{(2r)^p}{\mu(B)}\int_a^b g(x)^p w(x)\, dx
\end{split}    
\end{equation*}
so
\begin{equation}
\label{eq21}
\left(\mvint_B|u-u_B|^p\, d\mu\right)^{1/p}\leq
2\left(\mvint_B|u-u(a)|^p\, d\mu\right)^{1/p}\leq 4r\left(\mvint_B g^p\, d\mu\right)^{1/p}.
\end{equation}
Note that since the measure $\mu$ is not doubling, inequality \eqref{eq21} cannot hold with 
any exponent $q>p$ on the left hand side as otherwise we would arrive to a contradiction with Theorem~\ref{EPequiv}.
\end{example}

\begin{proof}[Proof of Theorem~\ref{EPequiv}]
The implication {\it (a)} $\Rightarrow$ {\it (b)} follows immediately from \cite[Theorem~1]{SMP2}. 
Note however, that the constant $\sigma$ in \eqref{eq20} might be larger than that in the $p$-Poincar\'e inequality (see \eqref{REW-1}). Thus
we will
focus on proving that {\it(a)} follows from {\it (b)}. 
To this end, suppose that $X$ satisfies the condition displayed in \eqref{eq20}. Making use of H\"older's inequality and the fact that
$1\leq p<q$, we may conclude that the $(q,p)$-Poincar\'e inequality in \eqref{eq20} implies that the space $(X,d,\mu)$ supports a $p$-Poincar\'e inequality (see \eqref{ppoin}).

There remains to show that the condition in \eqref{eq20} forces the measure $\mu$ to be doubling. Fix a ball $B:=B(x,r)$, $x\in X$, $r\in(0,\infty)$, and observe that specializing
\eqref{eq20} to the case when $B$ is replaced by $2\sigma B$ yields
\begin{equation}
\label{eq-JK2}
\left(\,\, \mvint_{2\sigma B} |u-u_{2\sigma B}|^q\, d\mu\right)^{1/q}
\leq 2\sigma rC_P\left(\,\, \mvint_{2\sigma^2 B} g^p\, d\mu\right)^{1/p},
\end{equation}
whenever $u\in L^{1}_{\rm loc}(X,\mu)$ and $g:X\to[0,\infty]$ is an upper gradient of $u$.
Since $p\geq1$, it follows from \eqref{eq-JK2} and
H\"older's inequality that,
\begin{align}
\label{eq-JK3}
\left(\,\, \mvint_{2\sigma B} |u|^q\, d\mu\right)^{1/q}
&\leq
\left(\,\, \mvint_{2\sigma B} |u-u_{2\sigma B}|^q\, d\mu\right)^{1/q}
+|u_{2\sigma B}|
\nonumber\\[4pt]
&\leq 2\sigma rC_P \left(\,\, \mvint_{2\sigma^2 B} g^p\, d\mu\right)^{1/p}+
\left(\,\, \mvint_{2\sigma B} |u|^p\, d\mu\right)^{1/p}.
\end{align}
We now define a collection of functions $\{u_j\}_{j\in\mathbb{N}}$
as follows:\ for each fixed $j\in\mathbb{N}$, let 
$r_j:=(2^{-j-1}+2^{-1})r$ and set $B_j:=B(x,r_j)$. Then
\begin{equation}\label{JG-1}
\frac{1}{2}r<r_{j+1}<r_j\leq\frac{3}{4}r,
\quad\forall\,j\in\mathbb{N}.
\end{equation}
For each $j\in\mathbb{N}$, let $u_j:X\to\mathbb{R}$ be the function
defined by setting for each $y\in X$,
\begin{eqnarray}\label{udef}
u_j(y):=
\left\{
\begin{array}{ll}
\,\qquad 1\quad &\mbox{if $y\in B_{j+1}$,}
\\[6pt]
\displaystyle\frac{r_j-d(x,y)}{r_j-r_{j+1}}
&\mbox{if $y\in B_j\setminus B_{j+1}$,}
\\[15pt]
\,\qquad 0 &\mbox{if $y\in X\setminus B_j$.}
\end{array}
\right.
\end{eqnarray}
Noting that $\displaystyle (r_j-r_{j+1})^{-1}=2^{j+2}r^{-1}$, 
a straightforward computation will show that $u_j$ is
$2^{j+2}r^{-1}$-Lipschitz on $X$ and that the function
$g_j:=2^{j+2}r^{-1}\chi_{B_j}$ is an upper
gradient of $u$, where $\chi_{B_j}$ denotes the characteristic function of the set $B_j$.
In particular, we have that $u_j\in L^{1}_{\rm loc}(X,\mu)$ and that the functions $u_j$ and $g_j$ satisfy \eqref{eq-JK3}.
Observe that for each fixed $j\in\mathbb{N}$, we have
(keeping in mind $\sigma\geq1$)
\begin{eqnarray}
\label{HD-1}
2\sigma rC_P\left(\,\, \mvint_{2\sigma^2 B} g_j^p\, d\mu\right)^{1/p}
= \sigma C_P2^{j+3}\left(\frac{\mu(B_j)}{\mu(2\sigma^2 B)}\right)^{1/p}
\leq \sigma C_P 2^{j+3}\left(\frac{\mu(B_j)}{\mu(2\sigma B)}\right)^{1/p}
\end{eqnarray}
and
\begin{eqnarray}
\label{HD-2}
\left(\,\, \mvint_{2\sigma B}|u_j|^p\, d\mu\right)^{1/p}
\leq \left(\frac{\mu(B_j)}{\mu(2\sigma B)}\right)^{1/p}.
\end{eqnarray}
Moreover, 
\begin{equation}
\label{HD-3}
\left(\,\, \mvint_{2\sigma B} |u_j|^{q}\, d\mu\right)^{1/q}\geq\left(\frac{\mu(B_{j+1})}{\mu(2\sigma B)}\right)^{1/q}.
\end{equation}
In concert, \eqref{HD-1}-\eqref{HD-3} and the extreme most
sides of the inequality in \eqref{eq-JK3}, give
\begin{align}
\label{HD-4}
\left(\frac{\mu(B_{j+1})}{\mu(2\sigma B)}\right)^{1/q}
&\leq\sigma C_P 2^{j+4}
\left(\frac{\mu(B_j)}{\mu(2\sigma B)}\right)^{1/p},
\quad\forall\,j\in\mathbb{N}.
\end{align}
Therefore
\begin{align}
\label{HD-42}
\mu(B_{j+1})^{1/q}
\leq\sigma C_P 2^{j+4}\frac{\mu(B_j)^{1/p}}{\mu(2\sigma B)^{(q-p)/pq}},
\quad\forall\,j\in\mathbb{N}.
\end{align}
With $\alpha:=q/p\in(1,\infty)$ we 
raise both sides of the 
inequality in \eqref{HD-42} to the power $p/\alpha^{j-1}$ in order to obtain
\begin{align}
\label{HD-5}
\mu(B_{j+1})^{1/\alpha^{j}}\leq 
2^{p(j+4)/\alpha^{j-1}}
\bigg(\frac{\sigma C_P}{\mu(2\sigma B)^{(q-p)/pq}}\bigg)^{p/\alpha^{j-1}}\mu(B_j)^{1/\alpha^{j-1}},
\quad\forall\,j\in\mathbb{N}.
\end{align}
If we let $P_j:=\mu(B_j)^{1/\alpha^{j-1}}$, then the
inequality in \eqref{HD-5} becomes
\begin{align}\label{HD-6}
P_{j+1}\leq 
2^{p(j+4)/\alpha^{j-1}}
\bigg(\frac{\sigma C_P}{\mu(2\sigma B)^{(q-p)/pq}}\bigg)^{p/\alpha^{j-1}}P_j,
\quad\forall\,j\in\mathbb{N},
\end{align}
which, together with an inductive argument and the fact that
$P_1\leq\mu(B)$, implies
\begin{align}\label{HD-7}
P_{j+1}&\leq P_1\prod_{k=1}^j
\left[2^{p(k+4)/\alpha^{k-1}}
\bigg(\frac{\sigma C_P}{\mu(2\sigma B)^{(q-p)/pq}}\bigg)^{p/\alpha^{k-1}}\right]
\nonumber\\[4pt]
&\leq \mu(B)\prod_{k=1}^j
\left[2^{p(k+4)/\alpha^{k-1}}
\bigg(\frac{\sigma C_P}{\mu(2\sigma B)^{(q-p)/pq}}\bigg)^{p/\alpha^{k-1}}\right],
\quad\forall\,j\in\mathbb{N}.
\end{align}

We claim that the product in \eqref{HD-7} converges as $j\to\infty$. Indeed, observe that
\begin{align}\label{QW-1}
\prod_{k=1}^\infty \bigg(\frac{\sigma C_P}{\mu(2\sigma B)^{(q-p)/pq}}\bigg)^{p/\alpha^{k-1}}
&=\bigg(\frac{\sigma C_P}{\mu(2\sigma B)^{(q-p)/pq}}\bigg)^{p\sum_{k=1}^\infty\alpha^{1-k}}
\nonumber\\[4pt]
&=
\bigg(\frac{\sigma C_P}{\mu(2\sigma B)^{(q-p)/pq}}\bigg)^{\frac{p\alpha}{\alpha-1}}
=\bigg(\frac{\sigma C_P}{\mu(2\sigma B)^{(q-p)/pq}}\bigg)^{\frac{pq}{q-p}},
\end{align}
and
\begin{equation}\label{QW-3}
\prod_{k=1}^\infty\big(2^{p(k+4)}\big)^{1/\alpha^{k-1}}
=2^{\sum_{k=1}^\infty p(k+4)\alpha^{1-k}}=:A(p,q)\in(0,\infty).
\end{equation}

On the other hand, it follows from \eqref{JG-1} that
\begin{equation}
0<\mu(2^{-1}B)^{1/\alpha^{j-1}}\leq
P_j=\mu(B_j)^{1/\alpha^{j-1}}\leq \mu(B)^{1/\alpha^{j-1}}<\infty,
\end{equation}
which, in turn, further implies $\displaystyle\lim\limits_{j\to\infty}P_j=1$.
Consequently, passing to the limit in \eqref{HD-7} yields
\begin{align}\label{ME12}
1&\leq \mu(B)\frac{\big(\sigma C_P\big)^{pq/(q-p)}}{\mu(2\sigma B)}A(p,q).
\end{align}
Hence,
\begin{align}\label{ME13}
\mu(2\sigma B)\leq \big(\sigma C_P\big)^{pq/(q-p)}A(p,q)\,
\mu(B).
\end{align}
Since $\sigma\geq1$, it follows that $\mu$ is doubling.
This finishes the proof of the second implication and,
in turn, the proof of the theorem.
\end{proof}

\begin{remark}
In the proof of the {\it (b)} $\Rightarrow$ {\it (a)} in Theorem~\ref{EPequiv}, one can compute the constant $A(p,q)$ appearing in \eqref{QW-3} by observing that (keeping in mind $\alpha=q/p$),
\begin{align}
{\sum_{k=1}^\infty p(k+4)\alpha^{1-k}}
&=p\sum_{k=1}^\infty \frac{k}{\alpha^{k-1}}+4p\sum_{k=1}^\infty \frac{1}{\alpha^{k-1}}
\nonumber\\[4pt]
&=\frac{p}{(1-1/\alpha)^2}+\frac{4p\alpha}{\alpha-1}
=\frac{pq^2}{(q-p)^2}+\frac{4pq}{q-p}.
\end{align}
Therefore,
$$A(p,q)=2^{\frac{pq^2}{(q-p)^2}+\frac{4pq}{q-p}}.$$
Hence, condition \eqref{eq20} implies that measure $\mu$ satisfies the following doubling condition:
$$
\mu(2B)\leq\Big(\sigma C_P2^{\frac{q}{(q-p)}+4}\Big)^{pq/(q-p)}\mu(B)\quad
\mbox{for all balls\,\,$B\subseteq X$.}
$$
\end{remark}

\noindent
{\bf Acknowledgements.} We would like to express our deepest gratitude to the referee who showed us Example~\ref{rem} and made other valuable comments.

\end{document}